\input ./style/arxiv-vmsta.cfg
\documentclass[numbers,compress,v1.0.1]{vmsta}

\volume{3}
\issue{4}
\pubyear{2016}
\firstpage{315}
\lastpage{323}
\doi{10.15559/16-VMSTA68}

\RequirePackage{bbold}
\newcommand{\lsum}{\sum\limits}
\newtheorem{thm}{Theorem}

\newtheorem{lemma}{Lemma}

\theoremstyle{definition}

\startlocaldefs
\newcommand{\lleft}{\left}
\newcommand{\rright}{\right}

\allowdisplaybreaks
\endlocaldefs

\begin{document}
\begin{frontmatter}

\title{An estimate for an expectation of the simultaneous renewal for time-inhomogeneous Markov chains}

\author{\inits{V.}\fnm{Vitaliy}\snm{Golomoziy}}\email{vitaliy.golomoziy@gmail.com}

\address{Taras Shevchenko National University of Kyiv,\\
Faculty of Mechanics and Mathematics,\\
Department of Probability Theory, Statistics and Actuarial
Mathematics,\\
60 Volodymyrska Street, City of Kyiv, 01033, Ukraine}

\markboth{V. Golomoziy}{Simultaneous renewal of time-inhomogeneous
Markov chains}

\begin{abstract}
In this paper, we consider two time-inhomogeneous Markov chains
$X^{(l)}_t$, $l\in\{1,2\}$, with discrete time on a general state space.
We assume the existence of some renewal set $C$ and investigate the
time of simultaneous renewal, that is, the first positive time when the
chains hit the set $C$ simultaneously.
The initial distributions for both chains may be arbitrary.
Under the condition of stochastic domination and nonlattice condition
for both renewal processes, we derive
an upper bound for the expectation of the simultaneous renewal time.
Such a bound was calculated for two time-inhomogeneous birth--death
Markov chains.
\end{abstract}



%
\received{28 November 2016}
%
\revised{6 December 2016}
%
\accepted{7 December 2016}
\publishedonline{23 December 2016}
\end{frontmatter}

\section{Introduction}

\subsection{Overview}
Simultaneous renewal is an important topic for a practical application
of Markov chains. Although it has its own value, for example, in
queuing theory, we are interested in its
investigation because
it plays an essential role in coupling construction, which can be used
to derive stability estimates of the $n$-step transition probabilities
and other results like the law of large numbers and limit theorems.

For example, in \cite{StabilityAndCoupling,MinorCondition}, we can find
how a stability estimate can be calculated using the coupling
method for two time-inhomogeneous Markov chains with discrete time on the
general state space.
Good examples of applications of the coupling method (for both
homogeneous and inhomogeneous Markov chains) are given in \cite{Douc,Douc2}.

It worth mentioning that the coupling construction for
time-inhomogeneous\break chains is slightly different from its classical setup
(see, e.g., \cite{Lindvall,Thorisson}).
Such a time-inhomogeneous coupling for general state space can be found
in \cite{StabilityAndCoupling}.
Its modification, called the maximal coupling, can be used for a
discrete space. More information about maximal coupling and its
application to stability in the time-homogeneous case can be found in
\cite{MaxCoupling1,MaxCoupling2}.

For maximal coupling and its application to stability in the
time-inhomogeneous case, see \cite{MaxCoupling4,MaxCoupling3,WidowPension}.

The crucial problem in the application of the results in the listed
papers was calculation of the expectation for the coupling moment
deriving from the simultaneous renewal.
But there were no good estimates for the expectation of a simultaneous
renewal for the time-inhomogeneous case.

For the time-homogeneous case, the paper \cite{HomogCouplingEstimate}
proposes such an estimate based on the Daley inequality (see \cite{Daley}).


In \cite{Coupling}, we derived conditions (see Thm.\ 3.1) that
guarantee that the expectation for the simultaneous renewal time is
finite. But there were no practical estimates for the expectation.



In \cite{Excess}, we derived an analogue of the Daley inequality that
is used in this paper.
The key condition for this inequality is a finiteness of the second
moment for the stochastic dominant of the original renewal process.
Thats why it is a~crucial condition for the estimate construction.



\subsection{Definitions and notation}
We consider two independent time-inhomogeneous Markov chains with
discrete time and general state space $(E, \mathfrak{E})$. We assume
that both chains are defined on the same
probability space $(\varOmega, \mathfrak{F}, \mathbb{P})$.
Denote these chains as $(X^{(1)}_n), (X^{(2)}_n), n\ge0$.
We use the following notation for the one-step transition probabilities:
\begin{align}
\begin{array}{c}
P_{lt}(x, A) = \mathbb{P}\bigl\{X^{(l)}_{t+1} \in A \big| X^{(l)}_t = x\bigr\},
\end{array} %
\end{align}
where $x \in E$ is an arbitrary element, $l \in\{1,2\}$, and
$A\in\mathfrak{E}$ is an arbitrary set.

We continue to use the definitions and notation from \cite{Coupling}.
We consider some set $C \in\mathfrak{E}$,
and our goal is to find an upper bound for the
expectation of the first time of visiting the set $C$ by both chains.

Define the renewal intervals
\begin{align}
\begin{array}{c}
\theta^{(l)}_0 = \inf\bigl\{t \ge0, X^{(l)}_t \in C\bigr\},\\[3pt]
\theta^{(l)}_n = \inf\bigl\{t \ge\theta^{(l)}_{n-1}, X^{(l)}_t \in C \bigr\},
\end{array} %
\end{align}
where $l\in\{1,2\},\ n\ge1$, and renewal times
\begin{align}
\tau^{(l)}_n = \lsum_{k=0}^n
\theta^{(l)}_k, \quad l \in\{1,2\}, \ n\ge1.
\end{align}

Then we can define the renewal probabilities
\begingroup
\abovedisplayskip=5pt
\belowdisplayskip=5pt
\begin{equation}
\label{gdef} g^{(t,l)}_{n} = \mathbb{P}\bigl\{
\theta^{(l)}_k = n \big| \tau^{(l)}_{k-1} = t
\bigr\},\quad \ l\in\{1,2\}, \ n\ge1.
\end{equation}
It is worth mentioning that, in general, $g^{(t,l)}_n$ also depends on
the value $x$ of $X^{(l)}_t$ which can hit different states inside $C$.
However, we will omit $x$ for simplicity.
We refer the reader to \cite{Coupling} for more details about
definition (\ref{gdef}).

Let us define the renewal sequence recursively:
\begin{equation}
\begin{array}{c}
u^{(t,l)}_0 = 1,\\[3pt]
\displaystyle u^{(t,l)}_n = \lsum_{k=0}^{n-1} u^{(t,l)}_k g^{(t+k, l)}_{n-k}.
\end{array} %
\end{equation}

The time of simultaneous hitting the set $C$ is defined as
\begin{equation}
T = \inf\bigl\{ t >0:\ \exists m, n,\ t = \tau^{(1)}_m =
\tau^{(2)}_n \bigr\}.
\end{equation}

The notion of the overshoot or excess is defined as follows:
\begin{equation}
R^{(l)}_n = \inf\bigl\{t>n:\ X^{(l)}_t
\in C\bigr\}, \quad l \in\{1,2\}.
\end{equation}
It is, in fact, the next time after $n$ when the chain $X^{(l)}$ hits
the set $C$.

\vspace*{-6pt}\section{Estimate for the expectation of the simultaneous hitting
time}\vspace*{-3pt}
First, we need put the condition on $u^{(t,l)}_n$ that guarantees its
separation from~0.
In the time-homogeneous case, this follows from the renewal theorem,
but for the time-inhomogeneous case, there is no such
theorem. Therefore, we need the following condition.

\textit{Condition A}.
There are a constant $\gamma> 0$ and a number $n_0\ge0$ such that,
for all $t,l$ and $n \ge n_0$,
\begin{equation}
u^{(t,l)}_n \ge\gamma.
\end{equation}
It is important that this condition also guarantees certain
``regularity'' of a chain in terms of periodicity. The periodic chains
obviously do not satisfy it.

There are various theorems that allow us to check Condition A in
practice. See, for example, \cite{CouplingExamples}, Theorems 4.1, 4.2,
4.3. We will later use some of them.

We need a condition of the stochastic domination in order to apply
Theorem 3.1 from \cite{Excess}.

\textit{Condition B}.
Distributions $(g^{(t,l)}_n)$ are stochastically dominated by some
sequence $(\hat g_n)$, $\hat g_n\ge0$, which means that
\begin{equation}
G^{(t,l)}_n = \lsum_{k > n} g^{(t,l)}_k
\le\hat G_n = \lsum_{k > n} \hat g_k
\end{equation}
and that the stochastic dominant $(\hat g_n)$ has finite first and
second moments
\begin{equation}
\begin{array}{c}
\displaystyle \hat\mu_1 = \lsum_{k \ge1} k \hat g_k = \lsum_{k\ge0} \hat G_k < \infty,\\[16pt]
\displaystyle \hat\mu_2 = \lsum_{k \ge1} k^2 \hat g_k < \infty.
\end{array} %
\end{equation}
\endgroup

The sequence $\hat G_n$ is nonincreasing because $\hat g_n \ge0$.

It is worth mentioning that we do not require $(\hat g_n)$ to be a
probability distribution, that is, the total mass $\sum_{k\ge1} \hat
g_k$ not necessarily equals $1$.

\begin{thm}\label{theorem1}
Assume that conditions (A) and (B) hold for the chains $(X^{(l)}_n)$,
$l\in\{1,2\}$, defined before and that the renewal sequences are
generated by them.
Then the expectation of the simultaneous hitting time for the set $C$
satisfies the inequality
\begin{equation}
\mathbb{E}[T] \le\mathbb{E}\bigl[\theta^{(1)}_0\bigr] +
\mathbb{E}\bigl[\theta ^{(2)}_0\bigr] + \frac{M}{\gamma},
\end{equation}
where
\begin{equation}
\label{m_def} M = \hat\mu_2 + \hat\mu_1 \bigl(
\gamma^{-1} + n_0\bigr).
\end{equation}
\end{thm}
\begin{proof}
Let us recall the notation from \cite{Coupling}:
\begingroup
\abovedisplayskip=5pt
\belowdisplayskip=5pt
\[
\nu_0 := \min\bigl\{j\ge1: \tau^1_j>n_0
\bigr\},
\]
\[
B_0 := \tau^1_{\nu_0},
\]
\[
\nu_1 := \min\bigl\{j\ge\nu_0: \tau^2_j-
\tau^1_{\nu_0}>n_0,\ \mbox{or}\ \tau
^2_j - \tau^1_{\nu_0}=0\bigr\},
\]
\[
B_1 := \tau^2_{\nu_1} - \tau^1_{\nu_0},
\]
and further on
\[
\nu_{2m} := \min\bigl\{j\ge\nu_{2m-1}:\ \tau^1_j
- \tau^2_{\nu_{2m-1}}>n_0\, ,\ \mbox{or}\
\tau^1_j - \tau^2_{\nu_{2m-1}}=0\bigr\},
\]
\[
B_{2m} := \tau^1_{\nu_{2m}} - \tau^2_{\nu_{2m-1}},
\]
\[
\xch{\nu_{2m+1} := \min\bigl\{j\ge\nu_{2m}:\ \tau^2_j - \tau^1_{\nu_{2m}}>n_0\, ,\ \mbox{or}\ \tau^2_j - \tau^1_{\nu_{2m}}=0\bigr\},}{\nu_{2m+1} := \min\bigl\{j\ge\nu_{2m}:\ \tau^2_j - \tau^1_{\nu_{2m}}>n_0\, ,\ \mbox{or}\ \tau^2_j - \tau^1_{\nu_{2m}=0 },}
\]
\[
B_{2m+1} := \tau^2_{\nu_{2m+1}} - \tau^1_{\nu_{2m}}.
\]
\endgroup

The moments $\nu_k$ are called coupling trials. Let us define $\tau=
\min\{n\ge1: B_n = 0\}$ and the sequence of sigma-fields
$\mathfrak{B}_n$, $n\ge0$, by
\[
\mathfrak{B}_n = \sigma\bigl[B_k,\nu_k,
\tau^l_j,\ k\le n,\ j\le\nu_n\bigr].
\]

We will use the same idea as in the Theorem 5.1 from \cite{Coupling}. \\
First, we assume that $\theta^{(2)}_0 = 0$, which means that the second
chain starts from the set~$C$.

The next representation of time $T$ is following directly from the definitions:
\begingroup
\abovedisplayskip=2pt
\belowdisplayskip=2pt
\begin{equation}
\label{T_ineq} T \le\theta^{(1)}_0 + \sum
_{n=0}^\tau B_n = \theta^{(1)}_0
+ \sum_{n\ge0} B_n \mathbb{1}_{\tau> n}.
\end{equation}
\endgroup

Using Lemma~\ref{lemma1} and the fact that $\{\tau> n-1\} \in\mathfrak
{B}_{n-1}$, we can derive the following inequality:
\begin{align}
E[B_n \mathbb{1}_{\tau> n} | \mathfrak{B}_{n-1}] &= E[B_n \mathbb{1}_{\tau\ge n} | \mathfrak{B}_{n-1}] + E[0 \mathbb{1}_{\tau= n} | \mathfrak{B}_{n-1}] \notag\\
&= \mathbb{1}_{\tau\ge n} E[B_n | \mathfrak{B}_{n-1}] \le\mathbb{1}_{\tau\ge n} M.\label{ineq1}
\end{align}

Lemma 8.5 from \cite{Coupling} implies
\begin{equation}
\mathbb{P}\{\tau> n\} \le(1-\gamma)^{n}.
\end{equation}

Taking the unconditional expectation of the both parts in (\ref{ineq1})
gives us
\begin{equation}
\mathbb{E}[B_n \mathbb{1}_{\tau> n}] \le M \mathbb{P}\{\tau> n
\} \le M (1-\gamma)^{n}.
\end{equation}

Applying this inequality to (\ref{T_ineq}), we have
\begin{align}
\mathbb{E}[T] & \le\mathbb{E}[\theta_0] + \mathbb{E}\Biggl[\sum_{n=0}^\tau B_n\Biggr] = \mathbb{E}[\theta_0] + \sum_{n\ge0} \mathbb{E}[B_n \mathbb{1}_{\tau> n}] \notag\\
&\le\mathbb{E}[\theta_0] + \lsum_{n\ge0} M (1-\gamma)^n = \mathbb{E}[\theta_0] + \frac{M}{\gamma}.\label{T_ineq2}
\end{align}

Now, we have to get rid of the assumption $\theta^{(2)}_0=0$. The same
calculations as in \cite{Coupling} after formula (20) give us
\[
\mathbb{E}[T] \le\mathbb{E}\bigl[\max\bigl(\theta^{(1)}_0,
\theta^{(2)}_0\bigr)\bigr] + \frac{M}{\gamma} \le\mathbb{E}
\bigl[\theta^{(1)}_0\bigr] + \mathbb{E}\bigl[\theta
^{(2)}_0)\bigr] + \frac{M}{\gamma}.\qedhere
\]
\end{proof}

\section{Application to the birth--death processes}
Consider two time-inhomogeneous processes $X^{(1)}$ and $X^{(2)}$ with
the following transition probabilities on the $t$th step:
\begin{equation}
P_{t} = \lleft( %
\begin{array}{cccccc}
\alpha_{t0}& 1-\alpha_{t0}& 0& 0& 0&\ldots\\
0& \alpha_{t1}& 0& 1-\alpha_{t1}& 0& \ldots\\
0& 0& \alpha_{t2}& 0& 1-\alpha_{t2}& \ldots\\
\ldots
\end{array} %
 \rright)
\end{equation}
and
\begin{equation}
Q_{t} = \lleft( %
\begin{array}{cccccc}
\beta_{t0}& 1-\beta_{t0}& 0& 0& 0&\ldots\\
0& \beta_{t1}& 0& 1-\beta_{t1}& 0& \ldots\\
0& 0& \beta_{t2}& 0& 1-\beta_{t2}& \ldots\\
\ldots
\end{array} %
 \rright).
\end{equation}

We would like to estimate the expectation applying Theorem~\ref
{theorem1}. So we have to check the regularity condition A and the
domination condition B.

We will need the second moment of the dominating distribution, which is
difficult to derive for chains $X^{(1)}$ and $X^{(2)}$. So the idea is
to construct a~domination sequence
based on some simple homogeneous Markov chain whose renewal sequence
is well studied and whose second moment can be calculated easily.
The closest chain similar to the birth--death chains we consider here
is a random walk on the half-line.

The domination sequence based on such a random walk is constructed in
Lem\-ma~\ref{lemma2}, and Lemma~\ref{lemma3} gives its first and second
moments that we need for Theorem~\ref{theorem1}.

Next, we will check regularity condition A. First, we assume that, for
every $t>0$, $g^{(l)}_1 = \alpha_{t0} > 0$, and
\begin{equation}
\label{gamma0_cond} \gamma_0 = \inf_t \{
\alpha_{t0}, \beta_{t0}\} > 0.
\end{equation}

We will use Corollary to Theorem 4.2 from \cite{CouplingExamples} in
order to check Condition~A.
It says that if $g^{(t)}_1 > 0$ and a domination sequence exists, then
Condition A holds.
Moreover, its proof (see \cite[p.~12]{CouplingExamples}, inequality for
$F(G)$) contains an estimate for $\gamma$:
\begin{equation}
\label{gamma_def} \gamma= \exp \bigl(\hat\mu \ln(\gamma_0)/
\gamma_0 \bigr).
\end{equation}
%

Finally, we can state the following result.
\begin{thm}\label{theorem2}
Assume that for chains with transition probabilities $P_t$, $Q_t$
defined before, condition (\ref{gamma0_cond}) holds and that there
exists $p$ that satisfies condition
(\ref{p_cond}) for both chains $X^{(1)}$ and $X^{(2)}$. If both chains
start from the zero state, then the expectation of their simultaneous
renewal satisfies the inequality
\begin{equation}
\mathbb{E}[T] \le\hat\mu_2/\gamma+ \hat\mu_1/
\gamma^2,
\end{equation}
where $\hat\mu_1, \hat\mu_2$ are defined in Lemma~\ref{lemma3}, and
$\gamma$ is defined in (\ref{gamma_def}).
\end{thm}
\begin{proof}
The statement of the theorem follows from Theorem~\ref{theorem1},
applied to chains $X^{(1)}$ and $X^{(2)}$ with domination sequence
constructed in Lemma~\ref{lemma2}, the constant $\gamma$
defined before, and the variables $\hat\mu_1$ and $\hat\mu_2$
calculated in Lemma~\ref{lemma3}.
\end{proof}

\section{Auxiliary results}
\begin{lemma}\label{lemma1}
We have the inequality
\begin{equation}
\label{cond_exp_ineq} %
\begin{array}{c}
\mathbb{E}[B_{n} | \mathfrak{B}_{n-1}] \le M,
\end{array} %
\end{equation}
for $n \ge1$, where $M$ is defined in (\ref{m_def}).
\end{lemma}

\begin{proof}
From Lemma 8.3 of \cite{Coupling} we can derive:
\begin{equation}
\label{cond_exp} %
\begin{array}{c}
\mathbb{E}[ B_{2n+1} | \mathfrak{B}_{2n}] = \displaystyle \lsum_{t} \mathbb{E}\bigl[R^{(2)}_{t+n_0}\bigr] \mathbb{1}_{\tau^{(1)}_{\nu_{2n}} = t}, \\[15pt]
\mathbb{E}[ B_{2n} | \mathfrak{B}_{2n-1}] = \displaystyle \lsum_{t} \mathbb{E}\bigl[R^{(1)}_{t+n_0}\bigr] \mathbb{1}_{\tau^{(2)}_{\nu_{2n-1}} = t}. \\
\end{array} %
\end{equation}

At the same time, Theorem 3.1 from \cite{Excess} gives us the inequality
\begin{equation}
\label{excess_estimate} \mathbb{E}\bigl[R^{(l)}_{m}\bigr] \le M,\quad m \ge0,\ l\in\{0,1\},
\end{equation}
taking into account the domination condition~B.

Putting (\ref{excess_estimate}) into formulas (\ref{cond_exp}) yields
the required result (\ref{cond_exp_ineq}).
\end{proof}

\begin{lemma}\label{lemma2}
Consider the following time-inhomogeneous birth--death Markov chain
$Z_t$ with the transition probabilities on the $t$-th step
\begin{equation}
P_t = \lleft( %
\begin{array}{cccccc}
\alpha_{t0}& 1-\alpha_{t0}& 0& 0& 0&\ldots\\
0& \alpha_{t1}& 0& 1-\alpha_{t1}& 0& \ldots\\
0& 0& \alpha_{t2}& 0& 1-\alpha_{t2}& \ldots\\
\ldots,
\end{array} %
 \rright)
\end{equation}
and the time-homogeneous random walk $\hat Z_t$ with the transition
probability matrix
\begin{equation}
\hat P = \lleft( %
\begin{array}{cccccc}
0 & 1& 0& 0& 0&\ldots\\
p& 0& 1-p& 0& 0& \ldots\\
0& p& 0& 1-p& 0& \ldots\\
\ldots,
\end{array} %
 \rright).
\end{equation}
Let
\[
C = \{0\},
\]
and let $g^{(t)}_n$ be a distribution of the first after $t$ returning
into $0$ for the chain~$Z$, which is in the zero state at the moment
$t$.

Assume that there exists some $p$ such that, for all $t,i,s,j$, the
following inequations hold:
\begin{equation}
\label{p_cond} %
\begin{array}{c}
p > 1/2,\\[3pt]
p(1-p) \ge(1-\alpha_{ti})\alpha_{sj}, \quad \forall t,s,i,j.\\
\end{array} %
\end{equation}
Denote by $f_n$ the renewal probability for the chain $\hat Z_t$ ($f_n$
is the probability of the first returning to $0$ for the chain $\hat Z$
started at $0$):
\begin{equation}
f_n = \mathbb{\hat P}\{\hat Z_0 = 0, \hat
Z_1 \neq0, \ldots, \hat Z_{n-1} \neq0, \hat Z_n =
0\},
\end{equation}
and let $\hat g_n = f_n/p$, $n > 1$, and $\hat g_1 = 1$.

Then the sequence $(\hat g_n)_{n\ge1}$ stochastically dominates
$g^{(t)}_n$, or, in other words,
\begin{equation}
G^{(t)}_k \le \hat G_k,
\end{equation}
where $G^{(t)}_n = \sum_{k > n} g^{(t)}_k$, $\hat G_n = \sum_{k > n}
\hat g_k$,
and $\mathbb{P}$, $\mathbb{E}$ and $\mathbb{\hat P}$, $\mathbb{\hat E}$
are the probabilities and expectations on the canonical probability
space for the chains $Z$ and $\hat Z$,
respectively.
\end{lemma}
\begin{proof}

First of all, notice that $(\hat g_n)$ is not a probability
distribution. But this is not a big problem since the domination
sequence in our construction does not have to be
a distribution.

We will show that
\begin{equation}
\label{gn_domination_ineq} \hat g_n \ge g^{(t)}_n,
\end{equation}
for all $t,n$.

Let us start with $n=1$:
\[
g^{(t)}_1 = \alpha_{t0} < 1 = \hat
g_1.
\]

For $n > 1$, consider the event
\[
A_{(2n)t} = \{Z_t = 0, Z_{t+1} \neq0, \ldots,
Z_{t+2n-1} \neq0, Z_{t+2n} = 0\}.
\]
It can be interpreted as a set of trajectories
$\omega=(\omega_1, \ldots,\omega_{2n}), ~{\omega_i \in\{+1, -1\}} $,
where $\omega_i = +1$ if $Z_{t+i}$ goes up and $\omega_i=-1$ otherwise.
It is clear that, in order to return back to $0$ at time $2n$, there
should be exactly $n$ steps up (the first one must be step up) and $n$
steps down.
It is worth mentioning that not every trajectory of length $2n$ that
has $n$ steps up and $n$ down belongs to $A_{(2n)t}$ because some of
them might visit $0$ before $2n+t$,
which is not acceptable for $A_{(2n)t}$. The exact number of such
trajectories in $A_{(2n)t}$ is unknown and not important for this proof.
What is important, is that each $\omega\in A_{(2n)t}$ corresponds to
the same trajectory for the chain $\hat Z$. This means that summing
$\mathbb{\hat P}\{\omega\}$
for all $\omega\in A_{(2n)t}$ gives the probability $f_{2n}$. Strictly
speaking, the chains~$Z$ and $\hat Z$ are defined on different
probability spaces, but there is an
obvious correspondence between the trajectories, and the difference is
only in the probabilities. So we can use the same symbol $\omega$ for both.

Since $\omega$ has exactly $n$ steps up and $n$ steps down, its
probability is a product of $n$ different $(1-\alpha_{t_i n_i})$ and
$n$ different $\alpha_{t_j, n_i+1}$, for
$i, j = \overline{1,n}$. Notice that some of $n_i$ can be the same.

This means that, after reordering, the probability of such $\omega$ can
be presented as
\begin{equation}
\mathbb{P}\{\omega\} = \prod \bigl((1-\alpha_{t_i n_i})
\alpha_{t_j
n_i+1} \bigr),
\end{equation}
for some $t_i, n_i, t_j$. We emphasize again that the terms in that
product may repeat, but this is not important for this proof.

At the same time, it follows from condition (\ref{p_cond}) that for any
indexes, $\alpha_{t n}(1-\alpha_{s m}) < p(1-p)$ and
\begin{equation}
\mathbb{P}\{\omega\} \le\bigl(p(1-p)\bigr)^{n-1} = \mathbb{\hat P}\{
\omega\}/p.
\end{equation}
Summing over all such $\omega$, we obtain (\ref{gn_domination_ineq})
for $n > 1$.
\end{proof}

\begin{lemma}\label{lemma3}
The sequence $(\hat g_n)$ defined in Lemma~\ref{lemma2} has the finite
first and second moments
\begin{equation}
\begin{array}{c}
\hat\mu_1 = 2/(2p-1) + 1 ,\\[3pt]
\hat\mu_2 = (2p-1)^{-1} \biggl( 2 + \dfrac{8(1-p)}{1-4p} \biggr) + 2/(2p-1) + 1.
\end{array} %
\end{equation}
\end{lemma}

\begin{proof}
First we note that since $\hat g_n = f_n/p$, $n>1$, and $\hat g_1 = 1
= 1+f_1$, we have $\hat\mu_1 = \mu_1/p +1$ and $\hat\mu_2 = \mu_{2}/p
+1$, where
$\mu_1$ and $\mu_2$ are the expectation and the second moment for the
probability distribution $f_n, n\ge1$.

The generating function $F(z)$ for the distribution $f_n$ equals
(see, e.g., \cite[Ch.~XIII]{Feller})
\begin{equation}
F(z) = \frac{1-\sqrt{1-4p(1-p)s^2}}{2(1-p)}.
\end{equation}

So, $ \mu_1 = F^\prime(1)$ and $\mu_2 = F^{\prime\prime}(1) + \mu_1$.
\end{proof}

\appendix

%

\end{document}